\documentclass[12pt]{amsart}
\usepackage{amsfonts,amsmath,amssymb,amsthm,latexsym,verbatim}
\newtheorem{theorem}{Theorem}[section]
\newtheorem{lemma}[theorem]{Lemma}

\newtheorem{corollary}[theorem]{Corollary}
\theoremstyle{remark}

\theoremstyle{definition}

\newtheorem{example}[theorem]{Example}

\DeclareMathOperator{\Hom}{{\mathrm{Hom}}}

\newcommand{\abs}[1]{|#1|}	
\newcommand{\Fq}{{\mathbb {F}_q}} 
\newcommand{\Fstar}{{\mathbb {F}_q^\times}} 
 
\newcommand{\Z}{{\mathbb {Z}}} 
\newcommand{\Q}{{\mathbb {Q}}}

\DeclareMathOperator{\Paley}{{\mathrm {Paley}}} 
\DeclareMathOperator{\diag}{{\mathrm {diag}}} 	
\newcommand{\allone}{{\mathbf 1}} 	

\begin{document}
\title[Smith and critical groups of Paley graphs]{The Smith and critical groups of Paley graphs}
\author{David B. Chandler, Peter Sin$^*$ and Qing Xiang}
\address{6 Georgian Circle, Newark, DE 19711\\USA}
\address{Department of Mathematics\\University of Florida\\ P. O. Box 118105\\ Gainesville, FL 32611\\ USA}
\address{Department of Mathematical Sciences, University of Delaware, Newark, DE 19716\\USA}
\date{}
\thanks{$^*$This work was partially supported by a grant from the Simons Foundation (\#204181 to Peter Sin)}



\begin{abstract}
There is a Paley graph for each prime power $q$ such that $q\equiv 1\pmod 4$. The vertex set is the field $\Fq$ and two vertices $x$ and $y$ are joined by an edge if and only if $x-y$ is a nonzero square of $\Fq$. We compute the Smith normal forms of the adjacency matrix and Laplacian matrix of a Paley graph.
\end{abstract}

\maketitle
\section{Introduction}
Let $\Gamma$ be a finite, simple, undirected and connected graph and let $A$ be the adjacency matrix of $\Gamma$  with respect to some fixed but arbitrary ordering of the vertex set of $\Gamma$. Let $D$ be the diagonal matrix whose $(i,i)$-entry is the degree of the $i^{\rm th}$ vertex. Then $L=D-A$ is called the {\it Laplacian matrix} of $\Gamma$. The matrices $A$ and $L$ represent endomorphisms (which will also be denoted by $A$ and $L$)  of the free abelian group on the vertex set. 
The structure of their cokernels as abelian groups is independent of the above ordering.  The cokernel of $A$ is called the {\it Smith group} $S(\Gamma)$, since its computation is equivalent to finding the Smith normal form of the matrix $A$. The endomorphism $L$ maps the sum of all vertices to zero, so the cokernel of $L$ is not a torsion group. The torsion subgroup $K(\Gamma)$ of the cokernel  of $L$ is called the {\it critical group} of $\Gamma$. It is known by Kirchhoff's matrix-tree theorem that the order of $K(\Gamma)$ is equal to the number of spanning trees of $\Gamma$.

The critical group of a graph arises in several contexts, for example in arithmetic geometry \cite{DL89}, in statistical physics \cite{DH} and in combinatorics \cite{Big}. There are also interpretations of the critical group in discrete dynamics (chip-firing games and abelian sandpile models, cf. \cite{LP}). We refer the reader to \cite{Lorenzini} for a discussion of these and other connections.

So far, there are very few families of graphs
for which the critical groups have been found, so it is of some interest to
compute the Smith and critical groups for some well-known families of graphs.
 
In this note we treat the  Paley graphs. 
Let $q=p^t$ be a fixed  prime power with $q\equiv 1\pmod 4$.
The Paley graph $\Paley(q)$ is defined by taking the field $\Fq$ as vertex set, 
with two vertices $x$ and $y$ joined by an edge if and only if $x-y$ is a nonzero square in $\Fq$. The degree of each vertex is $k=\frac{q-1}{2}$. 
Let $A$ denote the adjacency matrix and  $L=kI-A$  the Laplacian matrix.
Our main result is the computation of $S(\Paley(q))$  and $K(\Paley(q))$. 
There has been some earlier work in this direction. 
The structure of the $S(\Paley(q))$ was correctly conjectured in 
\cite[Ex. 4--8]{Rushanan1}.
In \cite{Lorenzini} the critical group of a conference graph  on a square-free number of vertices was calculated, and $\Paley(q)$ is such a graph when $q$ is a prime.
The  $p$-rank of the Laplacian of $\Paley(q)$ was first computed in \cite{BE}. 

Here is a brief outline of our method. 
We view $\Paley(q)$ as a Cayley graph, with the regular action
of the additive group of $\Fq$. Then, in \S\ref{eval}, we follow a standard method, 
applying the discrete Fourier transform while keeping track of coefficient rings,
to compute the Smith group and also the $p$-complementary part of the critical group.

A different approach is needed to compute the $p$-part of $K(\Paley(q))$. 
In \S\ref{jac} we study  the permutation action of the group $S$ of nonzero squares  
on $\Fq$ by multiplication and on the free $R$-module $R^{\Fq}$ with basis $\Fq$ over a
suitable extension ring $R$. The {\it isotypic component} of an $R$-free $RS$-module $M$ with respect to a character $\chi:S\to R^\times$ is defined to be the submodule
$\{m\in M \mid\,  sm=\chi(s)m\quad  \text{for all $s\in S$}\}$.   
The  $RS$-module $R^{\Fq}$ decomposes  into isotypic components of rank 2 (except for one of rank 3). Since $S$ preserves adjacency, these isotypic components are $A$-invariant. In the computation of the restriction of $A$ to each isotypic component, certain Jacobi sums arise naturally and the main problem
is reduced to determining the $p$-adic valuations of these Jacobi sums. The classical theorem of Stickelberger on Gauss sums gives the valuation for individual sums, but there remains the  problem of counting the number of sums with a given valuation. This counting problem is solved by the transfer matrix method in \S\ref{count}. It is also possible to
count directly, but the chosen method has the advantages of being  
systematic and of yielding immediately the rationality of the generating function.

\section{Eigenvalues and $p'$-torsion}\label{eval}
It is well known and easily checked that $\Paley(q)$ is a 
strongly regular graph and that its eigenvalues are $k=\frac{q-1}{2}$,
$r=\frac{-1+\sqrt{q}}{2}$ and $s=\frac{-1-\sqrt{q}}{2}$, with multiplicities $1$, $\frac{q-1}{2}$ and $\frac{q-1}{2}$, respectively. (See, for example, [8.1.1]\cite{BH}). 
Hence, $$\abs{S(\Paley(q))}=\det(A)=k\left(\frac k2\right)^{k},$$ where $A$ is the adjacency matrix of $\Paley(q)$. It follows that $\gcd(\abs{S(\Paley(q))}, q)=1$. Therefore we can use the diagonalization of $A$ by the character table of $(\Fq, +)$ to find the Smith normal form of $A$. 

Let $S$ be the set of nonzero squares in $\Fq$. We can view $\Paley(q)$ as a Cayley graph with connecting set $S$. Let $X$ be the complex character table of the additive group of $\Fq$ where the elements are ordered in the same way as for the rows and columns of $A$. The entries of $X$ lie in the ring $\Z[\zeta]$,
where $\zeta$ is a complex primitive $p$-th root of unity. As was first observed in  \cite{McWM}, we have the character orthogonality relation  $\frac1q X\overline{X}^t=I$ and
\begin{equation}\label{diag}
\frac1q XA\overline{X}^t=\diag(\psi(S))_{\psi},
\end{equation}
where $\psi$ runs over the additive characters of $\Fq$ and $\psi(S)=\sum_{y\in S}\psi(y)$. Thus, the $\psi(S)$ are the eigenvalues of $A$.
Since the eigenvalues of $A$ are all prime to $p$, the structure of $S(\Paley(q))$ can
be completely determined from (\ref{diag}). (See \cite[\S 3.2]{sin4}, or \cite[\S2]{CX}, where the same argument is used for difference sets.)
It suffices to determine the structure of the $\ell$-Sylow subgroup
for each prime $\ell$ different from $p$. Such an $\ell$ is unramified in $\Z[\zeta]$, and (\ref{diag}) 
can be interpreted as expressing the  equivalence of matrices
with entries in the localized ring $\Z[\zeta]_{(\ell)}$. This latter
ring is a principal ideal domain and the list of exact powers of $\ell$ dividing the
$\psi(S)$ is precisely the list of elementary divisors of $A$ with respect
to the prime $\ell$ (or $\ell$-elementary divisors for short).

By applying the above to each $\ell\neq p$, and  noting that $r$ and $s$
are coprime, with $rs=\frac{1-q}{4}$, we obtain the following result.

\begin{theorem} The Smith group of $\Paley(q)$ is isomorphic to
$\Z/2\mu\Z\oplus(\Z/\mu\Z)^{2\mu}$, where $\mu=\frac{q-1}{4}$.
\end{theorem}

From the eigenvalues of $A$, we easily obtain those of $L=kI-A$ (the Laplacian matrix of $\Paley(q)$), namely $0$, with multiplicity 1, and $\frac{(q+\sqrt{q})}{2}$ and $\frac{(q-\sqrt{q})}{2}$, each with multiplicity $\frac{q-1}{2}$. It follows from Kirchhoff's matrix-tree theorem that
$$|K(\Paley(q))|=\frac{1}{q}\left(\frac{q+\sqrt{q}}{2}\right)^{k}\left(\frac{q-\sqrt{q}}{2}\right)^k=q^{\frac{q-3}{2}}\mu^{k},$$
where $\mu=\frac{q-1}{4}$.

The $\ell$-elementary divisors of $L$  for primes $\ell\neq p$
can be found in exactly the same way as we found the elementary
divisors of $A$. We can therefore determine the subgroup $K(\Paley(q))_{p'}$
which is complementary to the Sylow $p$-subgroup of $K(\Paley(q))$.

\begin{theorem} Let $K(\Paley(q))=K(\Paley(q))_p\oplus K(\Paley(q))_{p'}$ be the decomposition
of the critical group of $\Paley(q)$ into its Sylow $p$-subgroup and $p$-complement. 
Then $K(\Paley(q))_{p'}\cong (\Z/\mu\Z)^{2\mu}$, where $\mu=\frac{q-1}{4}$.
\end{theorem}

The $p$-elementary divisors of $L$  remain to be computed and the rest of the paper is devoted to this task.

\section{Character sums and invariants}\label{jac}

We will adopt the same notation as in \S\ref{eval}. In order to find the $p$-elementary divisors of $L$, we will view 
the entries of $L$ as coming from some $p$-adic local ring. Let $q=p^t$ and $K=\Q_p(\xi_{q-1})$ be the unique unramified
extension of degree $t$ over $\Q_p$, the field of $p$-adic numbers, where $\xi_{q-1}$ is a primitive $(q-1)^{\rm
th}$ root of unity in $K$. Let $R=\Z_p[\xi_{q-1}]$ be the ring of integers in $K$. Then
$pR$ is the unique maximal ideal of $R$ and $R/pR\cong\Fq$. Let $T:\Fstar\to R^\times$
be the Teichm\"uller character of $\Fq$. Then $T$ is an $R$-valued multiplicative character of $\Fq$ of order $q-1$. 
Hence $T$ generates the cyclic group $\Hom(\Fstar, R^\times)$.

Let $R^\Fq$ be the free $R$-module with basis indexed by the elements of $\Fq$. For clarity, we
write the basis element corresponding to $x\in \Fq$ as $[x]$.
Then $\Fstar$ acts on $R^\Fq$, permuting the basis by field multiplication,
so that $R^\Fq$ decomposes as the direct sum $R[0]\oplus R^\Fstar$ of a trivial
module with the regular module for $\Fstar$. 
The regular module
$R^\Fstar$  decomposes further into the direct sum  of nonisomorphic 
$R\Fstar$-submodules of $R$-rank $1$, affording the characters $T^i$,
$i=0, 1,\ldots, q-2$.  A basis element for the component affording $T^i$ is
$$
e_i=\sum_{x\in\Fstar}T^i(x^{-1})[x].
$$
Here the subscript $i$ is read modulo $q-1$.
So $R^\Fq$ has  basis $\{e_i \mid i=1,\ldots, q-2\} \cup \{e_0, [0]\}$,
where we have separated out the basis for the $\Fstar$-fixed points.

Next consider the action of the subgroup $S$ of squares in $\Fstar$ on $R^\Fstar$. Then for $0\leq i\leq q-2$, $T^i$ and $T^{i+k}$ are equal on $S$.
For  $0<i\leq k-1$,  let $M_i$ be the $R$-submodule spanned by
$\{e_i, e_{i+k}\}$. Then  $M_i$ is the isotypic component for the character
 $T^i\vert_S$.
The submodule of $S$-fixed points on $R^\Fq$ has basis $\{[0],e_0,e_k\}$, but
since $e_0+[0]=\allone=\sum_{x\in\Fq}[x]$, we will use the basis
 $\{\allone,[0],e_k\}$ instead. Let $M_0$ denote this submodule of $R$-rank 3.
Then we have the decomposition 
\begin{equation}\label{Mdecomp}
R^\Fq=M_0\oplus M_1\oplus\cdots\oplus M_{k-1} 
\end{equation}
We can view $A$ and $L$ as endomorphisms of $R^\Fq$, with
$$A([x])=\sum_{s\in S}[x+s], \; x\in \Fq$$ and $$L([x])=k[x]-\sum_{s\in S}[x+s], \; x\in \Fq.$$ 
Since the action of $S$ preserves adjacency the maps $A$ and $L$ are $RS$-module endomorphisms. It follows  that $A$ and $L$ map each isotypic component 
$M_i$ to itself, for $0\leq i\leq k-1$, and so they preserve the
decomposition (\ref{Mdecomp}). Thus, with respect to the basis of $R^\Fq$
formed from the bases of the $M_i$,  the matrices of  the maps $A$ and $L$ 
have block diagonal form, with a $(2\times 2)$-block for each $M_i$, 
for $1\leq i\leq k-1$, and a $(3\times 3)$-block for $M_0$. We are therefore reduced to computing the elementary divisors of $L\vert_{M_i}$ and determining for a given $p$-power its total multiplicity, as $i$ varies.

The next two lemmas compute $L$ on each of the $M_i$. The character $T^k=T^{-k}$
of $\Fstar$ is the quadratic character and we denote it by $\chi$. Following the convention of Ax
\cite{AX}, $T^0$ is the character that maps all elements of
$\Fq$ to 1, while $T^{q-1}$ maps 0 to 0 and all other elements
to 1. Moreover nonprincipal characters take the value $0$ at $0$. 
With these conventions the characteristic function of $S$ is
\begin{equation}\label{basic}
\frac12(\chi +T^0-\delta_0), 
\end{equation}
where $\delta_0$ is 1 at 0 and zero elsewhere.  Also we will need Jacobi sums, which we define below.  For any two integers $a, b$,  we define the Jacobi sum $J(T^a, T^b)$ by
$$J(T^a, T^b)=\sum_{x\in \Fq}T^a(x)T^b(1-x).$$ 
From the above definition and our convention on $T^0$ and
$T^{q-1}$, we see that if $a\not\equiv 0$ (mod $q-1$), then
$$J(T^{a},T^0)=0, \;{\rm and}\; J(T^{a},T^{q-1})=-1.$$

\begin{lemma} Suppose $0\leq i\leq q-2$ and $i\neq 0$, $k$. Then
$$
L(e_i)=\frac12(qe_i-J(T^{-i},\chi)e_{i+k})
$$
\end{lemma}
\begin{proof}
We have
$$
\begin{aligned} 
A(e_i)&=\sum_{x\in\Fstar} T^i(x^{-1})\sum_{y\in S}[x+y]\\
&=\frac12\sum_{x\in\Fstar} T^i(x^{-1})\sum_{y\in\Fq}(\chi(y)+T^0(y)-\delta_0(y))[x+y]\\
&=\frac12\sum_{x\in\Fstar} T^i(x^{-1})\sum_{y\in\Fq}\chi(y)[x+y]\\
&+\frac12\sum_{x\in\Fstar} T^i(x^{-1})\sum_{y\in\Fq}[x+y]
-\frac12\sum_{x\in\Fstar} T^i(x^{-1})[x].
\end{aligned}
$$
The second sum is zero and the third is $-\frac12e_i$. For the first sum,
we have
\begin{equation}\label{subz}
\begin{aligned}
\sum_{x\in\Fstar} T^i(x^{-1})\sum_{y\in\Fq}\chi(y)[x+y]
=\sum_{z\in\Fq}\sum_{x\in\Fstar}T^i(x^{-1})\chi(z-x)[z].
\end{aligned}
\end{equation}

Then if $z\neq 0$, we have
$T^i(x^{-1})\chi(z-x)=T^i(z^{-1})\chi(z)T^{-i}((x/z))\chi(1-(x/z))$. The sum
over $x$ of these terms is the same over $\Fstar$ or $\Fq$ and is equal to
$$T^i(z^{-1})\chi(z)J(T^{-i},\chi)=T^{i+k}(z^{-1})J(T^{-i},\chi).$$
If $z=0$ then $\sum_{x\in\Fstar}T^i(x^{-1})\chi(-x)=0$. Thus, the outer sum
over all $z\in\Fq$ can be taken over $\Fstar$ and is equal 
to  $J(T^{-i},\chi)e_{i+k}$.  The lemma now follows since $L=kI-A$.
\end{proof}

\begin{lemma}
\begin{itemize}
\item[(i)] $L(\allone)=0$.
\item[(ii)] $L(e_k)=\frac12(\allone-q([0]-e_k))$.
\item[(iii)] $L([0])=\frac12(q[0]-e_k-\allone)$.
\end{itemize}
\end{lemma}
\begin{proof}
(i) is obvious and (iii) is straightforward. (ii) is proved by the same calculation
as in the previous lemma, using the fact that $J(\chi,\chi)=-\chi(-1)=-1$.
The only difference in the calculation is that in equation (\ref{subz})  the 
$z=0$ term is $(q-1)[0]$ instead of zero.

\end{proof}

\begin{corollary}\label{cor}
The Laplacian matrix $L$ is equivalent over $R$ to the diagonal matrix with diagonal entries $J(T^{-i},T^k)$,
for $i=1,\ldots, q-2$ and $i\neq k$, two $1$s and one zero. 
\end{corollary}


In view of Corollary~\ref{cor}, to compute the $p$-elementary divisors of $L$, we will need to know the $p$-adic valuations of Jacobi sums. Using Stickelberger's theorem on Gauss sums \cite{stick} (see \cite[p.~636]{dwork} for further reference) and the well-known relation between Gauss and Jacobi sums, we have

\begin{theorem}\label{jacobidiv}
Let $a$ and $b$ be integers such that $a\not\equiv 0$ (mod $q-1$), $b\not\equiv 0$ (mod $q-1$), and $a+b\not\equiv 0$ (mod
$q-1$). For any integer $x$, we use $s(x)$ to denote the sum of digits in the expansion of the least
nonnegative residue of $x$ modulo $(q-1)$ as a base $p$ number. Then
$$\nu_p(J(T^{-a},T^{-b}))=\frac{s(a)+s(b)-s(a+b)}{p-1},$$
where $\nu_p(J(T^{-a},T^{-b}))$ is the $p$-adic valuation of $J(T^{-a},T^{-b})$. In other words, the number of times that $p$ divides $J(T^{-a},T^{-b})$ is equal to the number of carries in the addition $a+b$ (mod $q-1$).
\end{theorem}

By Theorem~\ref{jacobidiv}, the $p$-adic valuation of $J(T^{-i},T^k)$ for $i\neq k$ is known to be equal to $c(i):=\frac{1}{p-1}(s(i)+\frac{t(p-1)}{2}-s(i+k))$, since
$s(k)=\frac{t(p-1)}{2}$. The number $c(i)$ can be interpreted as the number of carries, when adding the $p$-expansions of $i$ and $k$, modulo $q-1$. 
This will be formulated precisely in the next section.
It is already clear that
there are no elementary divisors of $L$ greater than $p^t$.  It is also easy to see that the $p$-rank of $L$ (i.e., the number of times that $p^0$ appears as an elementary divisor of $L$) equals $(\frac{p+1}2)^t$, since a necessary and sufficient condition for there to be no carries is that
each of the $p$-digits of $i$ be in the range from $0$ to $\frac{p-1}2$.
This was already shown by Brouwer and van Eijl \cite[p.336]{BE}.  Also, since
$c(i)+c(q-1-i)=t$, it follows from Corollary ~\ref{cor}
that the multiplicity of $p^t$ as an elementary divisor is $(\frac{p+1}2)^t-2$.
It remains to find the multiplicity of $p^\lambda$  for $1\leq \lambda\leq t-1$.
In order to find the multiplicity of $p^\lambda$ as an elementary divisor of $L$, 
we have to count the number of $i$ ranging from $1$ to $q-2$, $i\neq k$, such 
that adding $i$ to $k$ involves exactly $\lambda$ carries. 

\section{The Counting Problem}\label{count}

In order to finish our computations of the critical groups of $\Paley(q)$, we will solve the following counting 
problem using the transfer matrix method. For a discussion of the transfer matrix method and its various applications, we refer the reader to \cite[Section 4.7]{Stanley1}.\\

\noindent{\bf The Counting Problem.} Let $q=p^t$ be an odd prime power and $k=\frac{q-1}{2}$. For $1\leq \lambda\leq t-1$, what is the number of $i$, $1\leq i\leq q-2$,  $i\neq k$ such that adding $i$ to $\frac{q-1}2$ modulo $q-1$ involves exactly $\lambda$ carries?\\

We express integers $a$, $0\leq a\leq q-1$, in base $p$. That is,  we write
$$a=a_{t-1}p^{t-1}+a_{t-2}p^{t-2}+\cdots +a_1p+a_0,$$
where $0\leq a_i\leq p-1$ for all $i$. In what follows, we will simply write
 $a=a_{t-1}a_{t-2}\cdots a_1a_0$, and call the $a_i$'s the digits of $a$. 
Since we are going to add $a$ with $\frac{q-1}2$ modulo $q-1$, we will use 
the modular $p$-ary add-with-carry algorithm described in \cite[Theorem 4.1]{HHKWX}.
Let $a=a_{t-1}a_{t-2}\cdots a_1a_0$ and $b=b_{t-1}b_{t-2}\cdots b_1b_0$
be  integers in $\{1,2,\ldots, q-2\}$ such that $a+\frac{q-1}{2}=b$ modulo $q-1$.
By the modular $p$-ary add-with-carry algorithm (cf. \cite[Theorem 4.1]{HHKWX})
there is a unique {\it carry sequence}   
$c=c_{t-1}c_{t-2}\cdots c_1c_0$ with $c_i\in \{0,1\}$ and $c_{t}=c_{0}$ 
such that for all $0\leq i\leq t-1$, 
\begin{equation}\label{addwithcarry}
a_i+\frac{p-1}{2}+c_{i}=b_i+pc_{i+1}.
\end{equation}
By the number of carries we shall mean the number of $i$ with $c_i=1$. 

We will use the transfer matrix method to solve the above counting problem. 
This approach involves constructing a weighted digraph $G$, and changing the 
counting problem to that of counting closed walks in $G$ of certain length 
and weight. The above equations motivate us to construct the digraph $G$ on
 $[p]\times [2]$ (here $[p]=\{0,1,\ldots ,p-1\}$ and $[2]=\{0,1\}$) as 
follows: The vertices of $G$ are $({\alpha},{\gamma})\in [p]\times [2]$. There is an 
arc from $({\alpha},{\gamma})$ to $({\alpha}',{\gamma}')$ if and only if 
\begin{equation}\label{arc}
{\alpha}+\frac{p-1}{2}+{\gamma}={\beta}+p{\gamma}'
\end{equation}
for some ${\beta}\in [p]$. Furthermore if there is an arc $e$ from $({\alpha},{\gamma})$ 
to $({\alpha}',{\gamma}')$ we give the arc label ${\alpha}$ and weight $w(e):={\gamma}'$. 
By the weight of a walk in the digraph we shall mean the sum of the weights
of its arcs. 
Each walk of length $t$ in $G$ specifies an element $a=a_{t-1}a_{t-2}\cdots a_0\in\{0,1,\ldots,q-1\}$ by taking the first arc label to be $a_0$, the second to be $a_1$, etc.
In terms of the digraph $G$, the discussion leading up to equation (\ref{addwithcarry}) 
means that for each  $a\in\{1,\ldots,q-2\}\setminus\{\frac{q-1}{2}\}$, there is a unique closed walk of length $t$ in $G$ whose arc labels specify $a$, and that the weight of this walk is equal to the number of carries when adding $\frac{q-1}{2}$ to $a$ modulo $q-1$. 
There are some other closed walks of length $t$ in $G$, namely those whose
arc labels specify $a=0$, $\frac{q-1}{2}$ or $q-1$. 
We must check that these walks have weights  which are outside the range
$\{1,2,\ldots, t-1\}$ for $\lambda$ in our counting problem.
Once we have checked this, we  will know that our counting problem is equivalent to
counting, for each $1\leq\lambda\leq t-1$, the  closed  walks of length $t$ and  weight $\lambda$. If $a=0$, then $a_i=0$ for all $i$ and there is just one closed walk, 
of weight $0$. If $a=q-1$, then $a_i=p-1$ for all $i$ and there is again just one
closed walk, of weight $t$. Finally, for $a=\frac{q-1}{2}$ there are two
closed walks one of weight $0$, where in equation (\ref{arc}) we take
$\alpha=\frac{p-1}{2}$, $\beta=p-1$ and $\gamma=0$ for every arc,
and the other of weight $t$, where we take $\alpha=\frac{p-1}{2}$, $\beta=0$ and $\gamma=1$ for every arc.

Let $B$ be the adjacency matrix of the digraph $G$. More precisely, 
the rows and columns of $B$ are both indexed by $(\alpha,\gamma)\in [p]\times 
[2]$, and the entry $(({\alpha},\gamma), ({\alpha}', \gamma'))$ of ${B}$ is $0$ if there is no 
arc from $({\alpha},\gamma)$ to $({\alpha}',{\gamma}')$; is $1$ if there is an arc $e$ from 
$({\alpha},{\gamma})$ to $({\alpha}',{\gamma}')$ and $w(e)=0$; is $x$ if there is an arc $e$ 
from $({\alpha},{\gamma})$ to $({\alpha}',{\gamma}')$ and $w(e)=1$ (here $x$ is an indeterminate).  Note that since (\ref{arc}) does not involve ${\alpha}'$, the adjacency matrix 
${B}$ has only two distinct rows, each repeated $p$ times. More explicitly,
$${B}=
\left( 
{\begin{array}{rcc}&p&p\\
\frac{p+1}2 & \overbrace{\left\lbrace{\begin{array}{ccc} 1&\cdots&1\\
\vdots&&\vdots\\ 1&\cdots&1\end{array}}\right.}&
\overbrace{\begin{array}{ccc}
0&\cdots&0\\ \vdots&&\vdots\\ 0&\cdots&0 \end{array} }\\
\frac{p-1}2 & \left\lbrace \begin{array}{ccc}
0&\cdots&0\\ \vdots&&\vdots\\ 0&\cdots&0 \end{array} \right. &
\begin{array}{ccc}x&\cdots&x\\ \vdots&&\vdots\\ x&\cdots&x \end{array}\\
\frac{p-1}2 & \left\lbrace \begin{array}{ccc}
1 &\cdots&1\\ \vdots&&\vdots\\ 1&\cdots&1 \end{array} \right. &
\begin{array}{ccc}0&\cdots&0\\ \vdots&&\vdots\\ 0&\cdots&0 \end{array}\\
\frac{p+1}2 & \left\lbrace \begin{array}{ccc}
0&\cdots&0\\ \vdots&&\vdots\\ 0&\cdots&0 \end{array} \right. &
\begin{array}{ccc}x&\cdots&x\\ \vdots&&\vdots\\ x&\cdots&x \end{array}
\end{array}}
\right)
$$

If $\Psi=e_1e_2\cdots e_n$ is a walk in $G$, we define 
${\rm wt}(\Psi)=x^{w(e_1)+w(e_2)+\cdots +w(e_n)}$. 
Let $$C_G(n)=\sum_{\Psi}{\rm wt}(\Psi),$$ 
where the sum is extended over all closed walks of length 
$n$ in $G$. Then $C_G(n)=\sum_{m\geq 0}f(n,m)x^m$, where 
$f(n,m)$ is the number of closed walks of length $n$ and weight $m$. Let
$$F(z,x)=\sum_{n\geq 1}C_G(n)z^n.$$
By Corollary 4.7.3 of \cite[p. 501]{Stanley1}, we have
$$F(z,x)=-\frac{z \frac{\partial Q(z,x)}{\partial z}}{Q(z,x)},$$
where $Q(z,x)={\rm det}(I-z{B})$.  Using the above definition of $B$, we find that 
${\rm det}(I-zB)=1-\frac{p+1}{2}(1+x)z+pxz^2$. It follows that 
$$F(z,x)=\frac{U-2V}{1-U+V}=\frac{(U-V)-V}{1-(U-V)},$$
with $U=\frac{p+1}{2}(1+x)z$ and $V=pxz^2$. Extracting the coefficient of 
$x^mz^n$ in $F(z,x)$, we obtain
   $$f(n,m)=\sum_i\frac{n}{n-i}{n-i \choose i}{n-2i 
\choose m-i}(-p)^i\left(\frac{p+1}2\right)^{n-2i}.$$
Notice that $f(n,m)=f(n,n-m)$.

Summing up we have the following:

\begin{theorem}\label{final}
Let $q=p^t$ be a prime power congruent to 1 modulo 4.  Then the number of $p$-adic elementary divisors of 
$L(\Paley(q))$ which are equal to $p^\lambda,\ 0 \le \lambda<t$,  is
$$f(t,\lambda)=\sum_{i=0}^{\min\{\lambda,t-\lambda\}}\frac{t}{t-i}{t-i \choose i}{t-2i
\choose \lambda-i}(-p)^i\left(\frac{p+1}2\right)^{t-2i}.$$
The number of $p$-adic elementary divisors of $L(\Paley(q))$ which are equal to $p^t$ is 
$\left(\frac {p+1} 2\right)^t - 2$.
\end{theorem}
\begin{example} In \cite{Lorenzini}, $K(\Paley(25))$ is calculated directly by a computer.  
Here as an illustration of Theorem~\ref{final} we use the above formula to compute $K(\Paley(5^3))$ and $K(\Paley(5^4))$:
$$f(3,0)=3^3=27.$$
$$f(3,1)={3 \choose 1} \cdot 3^3 -\frac 3 2 {2 \choose 1}{1\choose 0}
\cdot 5 \cdot 3 = 36.$$
Therefore $$K(\Paley(5^3))\cong (\Z/31\Z)^{62}\oplus (\Z/5\Z)^{36}\oplus(\Z/25\Z)^{36}
\oplus(\Z/125\Z)^{25}.$$

$$f(4,0)=3^4=81.$$
$$f(4,1)={4\choose 1}\cdot 3^4-\frac43{3\choose 1}{2\choose0}\cdot5\cdot 3^2=144.$$
$$f(4,2)={4\choose2}\cdot 3^4-\frac43{3\choose 1}{2\choose 1}\cdot5\cdot 3^2
+\frac 42{2\choose2}{0\choose0}\cdot 5^2=176.$$
Therefore $$K(\Paley(5^4))\cong (\Z/156\Z)^{312}\oplus
(\Z/5\Z)^{144}\oplus(\Z/25\Z)^{176}\oplus(\Z/125\Z)^{144}\oplus(\Z/625\Z)^{79}.$$
\end{example}

\bibliographystyle{amsplain}

\begin{thebibliography}{10}

\bibitem{AX}
J.~Ax, \emph{The zeroes of polynomials over finite fields}, Amer. J. Math.
  \textbf{86} (1964), 255--261.

\bibitem{Big}
N.~L. Biggs, \emph{Chip-firing and the critical group of a graph}, J. Algebraic
  Combin. \textbf{9} (1999), no.~1, 25--45.

\bibitem{BH}
A.~E. Brouwer and W.~H. Haemers, \emph{Spectra of graphs}, Universitext,
  Springer, New York, 2012. \MR{2882891}

\bibitem{BE}
A.~E. Brouwer and C.~A. van Eijl, \emph{On the {$p$}-rank of the adjacency
  matrices of strongly regular graphs}, J. Algebraic Combin. \textbf{1} (1992),
  no.~4, 329--346. \MR{1203680 (94b:05217)}

\bibitem{CX}
D.~B. Chandler and Q.~Xiang, \emph{The invariant factors of some cyclic
  difference sets}, J. Combin. Theory Ser. A \textbf{101} (2003), no.~1,
  131--146. \MR{1953284 (2004c:05034)}

\bibitem{DH}
D.~Dhar, \emph{Self-organized critical state of sandpile automaton models},
  Phys. Rev. Lett. \textbf{64} (1990), no.~14, 1613--1616.

\bibitem{dwork}
B.~Dwork, \emph{On the rationality of the zeta function of an algebraic
  variety}, Amer. J. Math. \textbf{82} (1960), 631--648. \MR{0140494 (25
  \#3914)}

\bibitem{HHKWX}
T.~Helleseth, H.~D.~L. Hollmann, A.~Kholosha, Z.~Wang, and Q.~Xiang,
  \emph{Proofs of two conjectures on ternary weakly regular bent functions},
  IEEE Trans. Inform. Theory \textbf{55} (2009), no.~11, 5272--5283.
  \MR{2596975 (2011c:94090)}

\bibitem{LP}
L.~Levine and J.~Propp, \emph{What is ... a sandpile?}, Notices Amer. Math.
  Soc. \textbf{57} (2010), no.~8, 976--979.

\bibitem{DL89}
D.~Lorenzini, \emph{Arithmetical graphs}, Math. Ann. \textbf{285} (1989),
  no.~3, 481--501.

\bibitem{Lorenzini}
D.~Lorenzini, \emph{Smith normal form and {L}aplacians}, J. Combin. Theory Ser.
  B \textbf{98} (2008), no.~6, 1271--1300. \MR{2462319 (2010d:05092)}

\bibitem{McWM}
F.~J. MacWilliams and H.~B. Mann, \emph{On the {$p$}-rank of the design matrix
  of a difference set}, Information and Control \textbf{12} (1968), 474--488.
  \MR{0242696 (39 \#4026)}

\bibitem{Rushanan1}
J.~J. Rushanan, \emph{Topics in integral matrices and abelian group codes},
  ProQuest LLC, Ann Arbor, MI, 1986, Thesis (Ph.D.)--California Institute of
  Technology. \MR{2635072}

\bibitem{sin4}
P.~Sin, \emph{Smith normal forms of incidence matrices}, Sci. China Math.
  \textbf{56} (2013), no.~7, 1359--1371. \MR{3073372}

\bibitem{Stanley1}
R.~P. Stanley, \emph{Enumerative combinatorics. {V}ol. 1, second edition},
  Cambridge Studies in Advanced Mathematics, vol.~49, Cambridge University
  Press, Cambridge, 2012. \MR{1442260 (98a:05001)}

\bibitem{stick}
L.~Stickelberger, \emph{Ueber eine {V}erallgemeinerung der {K}reistheilung},
  Math. Ann. \textbf{37} (1890), no.~3, 321--367. \MR{1510649}

\end{thebibliography}

\def\cprime{$'$}
\providecommand{\bysame}{\leavevmode\hbox to3em{\hrulefill}\thinspace}
\providecommand{\MR}{\relax\ifhmode\unskip\space\fi MR }
\providecommand{\MRhref}[2]{%
  \href{http://www.ams.org/mathscinet-getitem?mr=#1}{#2}
}
\providecommand{\href}[2]{#2}

\end{document}